\documentclass{article}
\usepackage{amsmath,amsfonts,amsthm,amssymb}
\DeclareMathSymbol\nullset{\mathord}{AMSb}{"3F}

\def\set#1\endset{\{\,#1\,\}}
\let\iso=\simeq
\def\from{\mkern2mu\hbox{\rm :}\mkern2mu}
\def\comp{\mkern2mu\mathchoice%
        {\raise.35ex\hbox{$\scriptscriptstyle\circ$}}
        {\raise.35ex\hbox{$\scriptscriptstyle\circ$}}
        {\raise.14ex\hbox{$\scriptscriptstyle\circ$}}
        {\raise.14ex\hbox{$\scriptscriptstyle\circ$}}}

\def\choice#1,#2{\binom{#1}{#2}}
\def\kzeroxarg#1{k\langle {#1}\rangle_0}
\def\kzerox{\kzeroxarg{X}}
\def\kzeroxn#1{\kzeroxarg{#1}}
\def\ckzerox{k[X]_0}

\def\onekzerox{k[x]_0}
\def\onekonex{k[x]}
\def\konex{k\langle X\rangle}
\def\gf#1{\mathbb{F}_{#1}}

\def\gftwoz{\gf{2}[x]_0}
\def\gftwoxz{\gf{2}\langle X\rangle_0}
\def\unitaryw#1{\setbox0=\hbox{$W$}\hbox to \wd0{$W^u$\hss}_{#1}}
\def\unitaryv#1{\setbox0=\hbox{$V$}\hbox to \wd0{$V^u$\hss}_{#1}}
\let\cong=\equiv
\def\mod#1{\,\,(\text{mod}\,#1)}
\newcounter{parts}

\begin{document}

 \newtheorem{theorem}{Theorem}[section]
 \newtheorem{corollary}{Corollary}[section]
 \newtheorem{lemma}{Lemma}[section]
 \newtheorem{proposition}{Proposition}[section]
 {\newtheorem{definition}{Definition}[section]}


\title{Maximal $T$-spaces of a free associative algebra}
\author{C. Bekh-Ochir and S. A. Rankin}

\maketitle  

\begin{abstract}
 We study the lattice of $T$-spaces 
 of a free associative $k$-algebra over a nonempty set. It is shown 
 that when the field $k$ is infinite, then the lattice has a maximum element, and
 that maximum element is in fact a $T$-ideal.
 In striking contrast, it is then proven that when the field $k$ is finite, 
 the lattice of $T$-spaces has infinitely many 
 maximal elements (of which exactly two are $T$-ideals). Similar 
 results are also obtained for the free unitary associative $k$-algebras.
 The proof is based on the observation that there is a natural bijection
 between the sets of maximal $T$-spaces of the free associative $k$-algebras
 over a nonempty set $X$ and over a singleton set. This permits the
 transfer of results from the study of the lattice of $T$-spaces of
 the free associative $k$-algebra over a one-element set to the general 
 case. 
\end{abstract}

%

 
\section{Introduction}
 Let $k$ be a field, and let $A$ be an associative $k$-algebra. 
 A\hbox{.} V\hbox{.} Grishin introduced the concept of a $T$-space  of
 $A$ (\cite{Gr1}, \cite{Gr2}); namely, a linear subspace of $A$ that is
 invariant under the natural action of the transformation monoid  $T$ of
 all $k$-algebra endomorphisms of $A$. A $T$-space of $A$ that is also an
 ideal of  $A$ is called a $T$-ideal of $A$.  For any
 $H\subseteq A$, the smallest $T$-space  of $A$ containing $H$ shall be
 denoted by $H^S$, while the smallest $T$-ideal of $A$ that  contains
 $H$ shall be denoted by $H^T$. The set of all $T$-spaces of $A$ forms a
 lattice under the inclusion ordering, and we shall denote this 
 lattice by $L(A)$. 

 We shall let $\kzerox$ and $\konex$ denote the free, respectively free
 unitary, associative $k$-algebras on a set $X$. Our interest in this
 paper shall be the study of the maximal elements in the lattices
 $L(\kzerox)$ and $L(\konex)$ when $X$ is a nonempty set. We show that
 if $k$ is infinite, then the unique maximal $T$-ideal of $\kzerox$
 (that is, there is a maximum $T$-ideal) is also the unique maximal
 $T$-space. We then demonstrate that the story is strikingly different
 when $k$ is finite. We establish that there
 is a natural bijection between the sets of maximal $T$-spaces of
 $\kzerox$ and $\onekzerox$, which then allows us to focus on
 the study of the maximal $T$-spaces of $\onekzerox$. We prove that when
 $k$ is finite, there are infinitely many maximal $T$-spaces of
 $\onekzerox$ (and thus infinitely many maximal $T$-spaces of
 $\kzerox$). Our approach requires that we treat the case for $p>2$ and
 $p=2$ separately.

 We are able to adapt this analysis to determine that in the case
 of an infinite field $k$,  $\konex$ has a maximum proper
 $T$-ideal, and a maximum proper $T$-space (which of course contains the
 maximum proper $T$-ideal), so the situation is essentially the same as 
 that of the free associative $k$-algebra over $X$. In the case of a finite field, there
 is a slight difference, in that this time, there is a maximum
 proper $T$-ideal (as opposed to two maximal proper $T$-ideals in the
 non-unitary case). We then go on to prove that there are infinitely
 many maximal $T$-spaces of $\konex$ that contain this maximum proper
 $T$-ideal (actually, in this case, every maximal $T$-space contains the
 maximum $T$-ideal since the maximum $T$-ideal is $T^{(2)}$, and the proof
 of Proposiiton \ref{proposition: max contains t2} is also applicable for
 $T$-spaces of $\konex$).
 
\begin{lemma}\label{lemma: free have maximal}
 Let $A$ be a free associative or free commutative associative $k$-algebra
 on a nonempty set $X$. Then  every proper $T$-space ($T$-ideal) of $A$ is contained in a
 maximal $T$-space ($T$-ideal) of $A$.
\end{lemma}

\begin{proof}
 The proof for $T$-ideals is completely 
 analogous to the proof for $T$-spaces, and we shall present
 only the argument for $T$-spaces.
 Let $V$ be a proper $T$-space of $A$. Since $A$ is free on $X$, $V\cap X
 =\nullset$. Let $x\in X$ and consider the sub-partially ordered (poset) $P$ of $L(A)$ whose 
 elements are the $T$-spaces of $A$ that do not contain $x$ but do contain 
 $V$. Zorn's lemma may be  applied to $P$, so we conclude that $P$ has
 maximal elements. Let $M$ be any maximal element of $P$. If $M$ is not 
 maximal in $L(A)$, then there exists a proper $T$-space $U$ of $A$ that
 contains $M$, so $U\notin P$ and thus $x\in U$.  Since $x\in U$ and $A$
 is free on $X$, we conclude that $U=A$, which contradicts our choice of
 $U$. Thus $M$ is maximal in $L(A)$.
\end{proof}
 
We shall have frequent occasion to consider sets $X$ and $Y$ with $X\subseteq Y$.
In general, for $U\subseteq \kzerox$, when required for clarity, we shall write $U^T_X$,
rather than $U^T$, to denote the $T$-ideal of $\kzerox$ that is generated by $U$.

\begin{lemma}\label{lemma: extending t ideal}
 Let $X$ and $Y$ be nonempty sets with $X\subseteq Y$, and let $U\subseteq \kzerox$. 
 Then $U^T_X=U^T_Y\cap \kzerox$.
\end{lemma}

\begin{proof}
 Since every algebra endomorphism of $\kzerox$ extends to an algebra endomorphism of 
 $\kzeroxn{Y}$, it follows that $U^T_X\subseteq U^T_Y$, and thus $U^T_X\subseteq U^T_Y
 \cap \kzerox$. Accordingly, it suffices to prove that $U^T_Y\cap \kzerox\subseteq U^T_X$. 
 Let $u\in U^T_Y\cap \kzerox$. Then there exist $\alpha_i\in k$, $f_i\from \kzeroxn{Y}
 \to \kzeroxn{Y}$, $u_i\in U$, and $y_i,z_i\in \kzeroxn{Y}\cup \set 1\endset$ with 
 $u=\sum \alpha_iy_if_i(u_i)z_i$. Let $g\from \kzeroxn{Y}\to \kzeroxn{Y}$ be the map 
 determined by $x\mapsto x$ if $x\in X$, while $x\mapsto 0$ if $x\in Y-X$. As well, 
 let $\iota\from \kzerox\to \kzeroxn{Y}$ be the map determined by $\iota(x)=x$ for 
 each $x\in X$. Then since $u\in \kzerox$, we have $u=g(u)=\sum \alpha_i g(y_i)g\comp
 f_i(u_i)g(z_i)$, and since $u_i\in U$, we have $u_i=\iota(u_i)$, so 
 $u=\sum \alpha_i g(y_i)g\comp f_i\comp \iota(u_i)g(z_i)$. Since $g\comp f_i\comp\iota\from 
 \kzerox\to\kzerox$, $u_i\in U$, and $g(y_i),g(z_i)\in \kzerox$ for every $i$, it 
 follows that $u\in U$.
\end{proof} 

\begin{proposition}\label{proposition: comparison of t-ideals}   
 Let $X$ and $Y$ be nonempty sets with $X\subseteq Y$. The map $U\mapsto U_Y^T$ from the lattice of
 $T$-ideals of $\kzerox$ into the lattice of $T$-ideals of $\kzeroxn{Y}$ is injective, and moreover, if $U_Y^T$ is
 a maximal $T$-ideal in $\kzeroxn{Y}$, then $U$ is a maximal $T$-ideal in $\kzerox$. If $X$ is infinite,
 then the map is surjective and thus a lattice isomorphism.
\end{proposition}

\begin{proof}
 By Lemma \ref{lemma: extending t ideal}, the map is injective. 
 Suppose that $U$ is a $T$-ideal of $\kzerox$ such that $U_Y^T$ is a maximal $T$-ideal of $\kzeroxarg{Y}$.
 Let $M$ be a maximal 
 $T$-ideal of $\kzerox$ with $U\subseteq M$. By Lemma \ref{lemma: extending t ideal},
 $M=M^T_Y\cap \kzerox$, so $M^T_Y\ne \kzeroxn{Y}$. Since $U\subseteq M$, we have 
 $U^T_Y\subseteq M^T_Y$ and $U^T_Y$ is maximal, so $U^T_Y=M^T_Y$. Thus 
 $U=U^T_Y\cap \kzerox=M^T_Y\cap \kzerox=M$, and so $U$ is maximal, as required.
 Finally, suppose that $X$ is infinite, and let $V$ be a $T$-ideal of $\kzeroxn{Y}$.
 Then $U=V\cap\kzerox$ is a $T$-ideal of $\kzerox$ and $U_Y^T\subseteq V$. We claim
 that $V\subseteq U_Y^T$. Let $f\in V$. Then since $X$ is infinite, there exists a 
 $k$-algebra automorphism $\sigma$ of $\kzeroxn{Y}$ such that $\sigma(f)\in \kzerox$.
 Since $V$ is a $T$-ideal of $\kzeroxn{Y}$, we have $\sigma(f)\in V$ and thus
 $\sigma(f)\in U$. But then $f=\sigma^{-1}(\sigma(f))\in U_Y^T$, as required.
\end{proof}

\begin{definition}
 For any nonempty set $X$, let $Z_X=\set xy\endset^T$ if $|X|>1$, otherwise let
 $Z_X=\set x^2\endset^T$, where $X=\set x\endset$.
\end{definition}
 
\begin{corollary}\label{corollary: gen of [S] z result}
 For any nonempty set $X$, $Z_X$ is a maximal $T$-ideal of $\kzerox$, and if
 $k$ is infinite, then $Z_X$ is in fact the maximum $T$-ideal of $\kzerox$.
\end{corollary}

\begin{proof}
 Let $Y$ be an infinite set with $X\subseteq Y$. By Theorem 3 of \cite{S},
 $Z_Y$ is a maximal $T$-ideal of $\kzeroxn{Y}$ and in fact, is the
 maximum $T$-ideal of $\kzeroxn{Y}$ if $k$ is infinite. By Lemma \ref{lemma: extending t ideal},
 $Z_X=Z_Y\cap\kzerox$, and thus by Proposition \ref{proposition: comparison of t-ideals}, $Z_X$ is
 a maximal $T$-ideal of $\kzerox$. If $k$ is infinite and $U$ is a maximal $T$-ideal of $\kzerox$,
 then $U^T_Y\subseteq Z_Y$ and so $U=U^T_Y\cap \kzerox\subseteq Z_Y\cap \kzerox=Z_X$. As $U$ 
 is maximal, we must have $U=Z_X$.
\end{proof} 

 In the proof of Corollary \ref{corollary: gen of [S] z result}, it was observed that
 $Z_X=Z_Y\cap \kzerox$. Consequently, in a bid to simplify notation, from now on for 
 any nonempty set $X$, we shall write $Z$ in place of $Z_X$ when no confusion can 
 result from doing so.

\begin{definition}
 Let $X$ be any nonempty set. In $\kzerox$, if $|X|=1$, let $T^{(2)}=\set 0\endset$, 
 otherwise let $x,y\in X$ with $x\ne y$ and set $T^{(2)}=\set [x,y]\endset^{T_X}$.
\end{definition}

\begin{definition}
 Let $X$ be a nonempty set, and let $k$ be a finite field of order $q$. For any 
 $x\in X$, let $W_0=T^{(2)}+\set x-x^q\endset^T_X$.
\end{definition}

 Theorem 3 of \cite{S} also implies that if $X$ is infinite and $k$ is
 finite of order $q$, then $W_0$ is a maximal $T$-ideal of $\kzerox$,
 and furthermore, that $Z$ and $W_0$ are the only maximal $T$-ideals of
 $\kzerox$. 

 We remark that when we are considering nonempty sets $X\subseteq Y$ and we refer to $T^{(2)}$, 
 we shall rely on the context to determine whether we mean $T^{(2)}\subseteq \kzerox$ or 
 $T^{(2)}\subseteq \kzeroxarg{Y}$.

\begin{corollary}\label{corollary: finite field version}
 Let $k$ be a finite field of order $q$, and let $X$ be a nonempty set. 
 Then $Z$ and $W_0$ are maximal $T$-ideals of $\kzerox$, and these are 
 the only maximal $T$-ideals of $\kzerox$.
\end{corollary}

\begin{proof}
 Let $Y$ be an infinite set containg $X$. We observe that
 for $x\in X$, $(T^{(2)}+\set x-x^q\endset^T_X)^T_Y =(T^{(2)})^T_Y
 +(\set x-x^q\endset^T_X)^T_Y= T^{(2)}+\set x-x^q\endset^T_Y$. By Theorem 3 of \cite{S} 
 for countably infinite $Y$ in combination with Proposition \ref{proposition: comparison of t-ideals}
 for arbitrary infinite $Y$, $T^{(2)}+\set x-x^q\endset^T_Y$
 is maximal in $\kzeroxn{Y}$. Thus $T^{(2)}+\set x-x^q\endset^T_X$
 is maximal in $\kzerox$. Now, if $U$ is maximal in $\kzerox$, then $U^T_Y$ is 
 contained in either $Z_Y$, in which case $U\subseteq Z_Y\cap \kzerox=Z_X$ and 
 thus $U=Z_X$, or else $U^T_Y$ is contained in $T^{(2)}+\set x-x^q\endset^T_Y$,
 in which case $U$ is contained in $(T^{(2)}+\set x-x^q\endset^T_Y)\cap \kzerox
 =T^{(2)}+\set x-x^q\endset^T_X=W_0$ and so $U=W_0$.
\end{proof} 

\begin{proposition}\label{proposition: max contains t2}
 Let $X$ denote any nonempty set. Then every maximal $T$-space of 
 $\kzerox$ contains $T^{(2)}$.
\end{proposition}

\begin{proof}
 There is nothing to prove if $|X|=1$, so suppose that $|X|>1$.
 Let $U$ be a maximal $T$-space of $\kzerox$, and suppose that $U$ does not contain
 $T^{(2)}$. Then $U+T^{(2)}=\kzerox$, and so 
 for any $x\in X$, $x=f+g$ for some essential
 $f\in U$ and essential $g\in T^{(2)}$. But then $g$ depends only on
 $x$, and so $g=0$. Thus $x\in U$, which means that $U=\kzerox$. Since
 this is not the case, it follows that $T^{(2)}\subseteq U$.
\end{proof} 

\begin{proposition}\label{proposition: infinite field t-space}
 Let $X$ denote any nonempty set. If $k$ is infinite, then every 
 proper $T$-space of $\kzerox$ is contained in $Z$.
\end{proposition} 

\begin{proof}
 Let $V$ be a $T$-space of $\kzerox$ that is not contained in $Z$. Then
 there exists $f\in V$ with nonzero linear term. Since $k$ is infinite,
 each multihomogeneous component of $f$ belongs to $V$, so $V$ contains
 some $x\in X$. Thus $V=\kzerox$.
\end{proof}

\section{$k$ a finite field}
 We now turn our attention to the case when $k$ is a finite field, say
 of order $q$ and characteristic $p$. Let $X$ be a nonempty set. It will
 be useful to introduce the following notion.

\begin{definition}
 Let $k$ be a finite field of order $q$. 
 Then for monomials $u_i\in \kzerox$ and $\alpha_i\in k$, $1\le i\le t$,
 $f=\sum_{i=1}^t \alpha_i u_i$ shall be said to be $q$-homogeneous if 
 for each $x\in X$ and each $i,j$ with $1\le i,j\le t$,  $\deg_x(u_i)
 \cong \deg_x(u_j)\mod{\mkern 4mu q-1}$.
\end{definition}

 The usual Vandermonde (homogeneity) argument can then be used to prove
 that if $k$ is a field of order $q$ and $V$ is a $T$-space of $\kzerox$,
 then each $q$-homogeneous component of each element of $V$ is also an
 element of $V$.

 It was proven in Corollary \ref{corollary: finite field version} that 
 $Z$ and $W_0$ are the only maximal $T$-ideals of $\kzerox$.

\begin{proposition}\label{proposition: max t-ideals are max t-spaces}
 $Z$ and $W_0$ are maximal $T$-spaces of $\kzerox$.
\end{proposition}

\begin{proof}
 First, suppose that $V$ is a $T$-space of $\kzerox$ with $Z\subsetneq
 V$, and let $f\in V-Z$. Since $Z\subset V$, we may
 assume that $f$ is linear, say $f=\sum_i \alpha_ix_i$ for some
 $x_i\in X$ and $\alpha_i\in k^*=k-\set 0\endset$. Let $x\in X$ be one of
 the variables that appears in $f$, and let $\sigma\from \kzerox\to\kzerox$
 be the $k$-algebra map determined by sending $x\mapsto x$ and $y\mapsto 0$
 for all $y\in X-\set x\endset$. Then $\sigma(f)$ is a nonzero scalar multiple
 of $x$ and thus $x\in V$, establishing that $V=\kzerox$. This proves that
 $Z$ is a maximal proper $T$-space of $\kzerox$.
 
 Now suppose that $V$ is a $T$-space of $\kzerox$ with $W_0\subsetneq V$,
 and let $f\in V-W_0$. We may assume that $f$ is
 essential, depending on the variables $x_1,x_2,\ldots,x_n\in X$. Since
 $T^{(2)}\subseteq W_0$, we may further assume that $f$ is a linear
 combination of monomials, each of the form $x_1^{i_1}x_2^{i_2}\cdots
 x_n^{i_n}$. Additionally, since $x^q-x\in W_0$,  for any $x\in X$, we
 may assume that each exponent $i_j<q$. Now, of all such elements of
 $V-W_0$, let us suppose that $f$ is such that the number of different
 monomials is least. We claim that $f$ is (a scalar multiple of) a 
 monomial. For suppose not. Then for some index $i$, there are two
 monomial summands of $f$ in which the degree of $x_{i}$ is different.
 Again, since $T^{(2)}\subseteq W_0$, we may assume that $i=n$. For each $j$
 such that there is a monomial in which the degree of $x_n$ is $j$, let
 $g_j$ denote the sum of all such monomials (with their coefficients)
 with $x_n^j$ factored out, otherwise let $g_j=0$. Then
 $f=\sum_{i=1}^{r} g_ix_n^i$, where $r<q$ is the degree of $x_n$ in $f$.
 We may apply the Vandermonde argument (see for example the proof of
 Proposition 4.2.3 of \cite{Dr}) to conclude that for each $i$
 with $g_i\ne 0$, $g_ix_n^{i}\in V$. Since there are at least two
 distinct values of $i$ with $g_i\ne 0$, we have a contradiction to the
 choice of $f$. Thus there exists a monomial $x_1^{i_1}\cdots
 x_n^{i_n}\in V$, and so there exists $t$ such that for $x=x_1$,
 $x^t\in V$. Again, since $x^q-x\in W_0$, we may assume that $t<q$. If
 $p$ divides $t$, say $t=lp^s$ with $(l,p)=1$, then the substitution
 $x\mapsto x^{p^{m-s}}$, where $q=p^m$, establishes that $(x^q)^l\in V$ 
 and so $x^l\in V$, and we note that $l<t$, so in such a case, $t$ is
 not minimal with respect to $x^t\in V$. On the other hand, suppose 
 that $(t,p)=1$. Then $(x+x^2)^t=\sum_{i=0}^t \choice t,i
 x^{t+i}\in V$, and 
 the coefficient of $x^{t+1}$ is $\choice t,1=t\ne 0$.
 Note that for $0\le i\le t<q$, $t+i<t+q=t+1+(q-1)$, so no other power
 of $x$ that appears in the expansion of $(x+x^2)^t$ has exponent
 congruent to $t+1\mod{q-1}$. Thus we may apply the Vandermonde argument
 to conclude that $x^{t+1}\in V$. Suppose now that $t$ is minimal with
 respect to $x^t\in V$. Then by our earlier observation,  $(t,p)=1$, and
 so there exists $s\ge0$ with $sp<t<(s+1)p$. We may repeatedly apply the 
 above observation to conclude that $x^{(s+1)p}\in V$. But then the
 substitution $x\mapsto x^{p^{m-1}}$ establishes  that $x^{s+1}\in V$.
 By the minimality of $t$, we then have $sp<t\le s+1$, and thus $s=0$,
 which yields $x\in V$. Thus $V=\kzerox$.
\end{proof}

 Unlike the situation for an infinite field, when $k$ is finite, 
 not every maximal $T$-space of $\kzerox$ is a maximal $T$-ideal
 of $\kzerox$,
 as we shall soon see. 

 We shall denote the free commutative associative algebra on $X$ by
 $\ckzerox$. Note that $\ckzerox\iso \kzerox/T^{(2)}$.

\begin{proposition}\label{proposition: lattice iso}
 The map $u\in \kzerox$ to $u+T^{(2)}\in \kzerox/T^{(2)}\iso \ckzerox$
 induces a lattice isomorphism between the lattice of $T$-spaces of
 $\kzerox$ that contain $T^{(2)}$ and the lattice of $T$-spaces of
 $\ckzerox$. 
\end{proposition}

\begin{proof}
 Since $T^{(2)}$ is a $T$-ideal of $\kzerox$, for every algebra
 endomorphism $f$ of $\kzerox$, there exists a unique algebra
 endomorphism $\overline{f}$ of $\ckzerox$ with $\pi_2\comp
 f=\overline{f}\comp \pi_2$, where
 $\pi_2\from\kzerox\to\kzerox/T^{(2)}\iso \ckzerox$ is given by
 $\pi_2(u)=u+T^{(2)}$. Conversely, since $\kzerox$ is the free
 associative algebra on the set of generators $X$, it follows that for
 every algebra homomorphism $\overline{f}\from\ckzerox\to\ckzerox$, 
 there exists an algebra homomorphism $f\from \kzerox\to\kzerox$ with
 $\pi_2\comp f=\overline{f}\comp\pi_2$. Thus if $U$ is a $T$-space of
 $\ckzerox$, then $\pi_2^{-1}(U)$ is a $T$-space of $\kzerox$ that
 contains $T^{(2)}$. As well, if $U\subseteq \kzerox$ is a $T$-space of
 $\kzerox$, then $\pi_2(U)$ is a $T$-space of $\ckzerox$, and
 $\pi_2^{-1}(\pi_2(U))=U+T^{(2)}$, so if $T^{(2)}\subseteq U$,
 $\pi_2^{-1}(\pi_2(U))=U$. This establishes the map given by $u\mapsto
 u+T^{(2)}$ determines a bijective mapping between the set of all
 $T$-spaces of $\kzerox$ that contain $T^{(2)}$ and the set of all
 $T$-spaces of $\ckzerox$, and the lattice properties of this mapping
 follow immediately. 
\end{proof}

\begin{corollary}\label{corollary: bijective mapping on maximal t-spaces}
 The maximal $T$-spaces of $\kzerox$ are in bijective correspondence
 with the maximal $T$-spaces of $\ckzerox$.
\end{corollary}

\begin{proof}
 By Proposition \ref{proposition: lattice iso}, the lattice of $T$-spaces of $\kzerox$ that
 contain $T^{(2)}$ is isomorrphic to the lattice of $T$-spaces of
 $k[X]_0$, and by Proposition \ref{proposition: max contains t2}, every maximal $T$-space
 of $\kzerox$ contains $T^{(2)}$.
\end{proof}
 
 Thus the study of the maximal $T$-spaces of $\kzerox$ can be reduced
 (if one can think of this as a reduction) to the study of the maximal
 $T$-spaces of $\ckzerox$.

 Recall that for any $k$-algebra $A$, $L(A)$ denotes the lattice of all
 $T$-spaces of $A$. We shall let $M(A)$ denote the set of maximal
 $T$-spaces of $A$. Note that by Lemma \ref{lemma: free
 have maximal}, if $A$ is a free associative (commutative or otherwise)
 $k$-algebra, then  $M(A)$ is not empty.

 Let $x\in X$, and let $\pi\from \kzerox\to \onekzerox=xk[x]$, the free
 associative algebra on the generator $x$, denote the algebra
 homomorphism determined by mapping each $z\in X$ to $x$. Then for
 each $T$-space $U$ of $\kzerox$, $\pi(U)$ is a $T$-space of
 $\onekzerox\subseteq \kzerox$, and  $\pi(U)\subseteq U$. We note that 
 $\pi\from \kzerox\to \onekzerox$ induces a poset map from $L(\kzerox)$
 to $L(\onekzerox)$ (which we shall also denote by $\pi$). Now, there
 is a natural poset map $\omega\from L(\onekzerox)\to L(\kzerox)$ given 
 by $\omega(V)=V^S$, where $V$ is a $T$-space of $\onekzerox$ and $V^S$
 is the $T$-space of $\kzerox$ that is generated by  $V\subseteq
 \onekzerox\subseteq\kzerox$. Evidently, $\omega(\pi(V))\subseteq V$ for
 every $V\in L(\kzerox)$, while  $\pi(\omega(V))=V$ for every $V\in
 L(\onekzerox)$. In particular, we note that $\pi$ is surjective.

 \begin{lemma}\label{lemma: one variable guys}
 Let $V$ be a $T$-space of $\onekzerox$. Then the subset of $L(\kzerox)$
 that consists of all $T$-spaces $Y$ of $\kzerox$ for which $\pi(Y)=V$
 is an interval with minimum element $\omega(V)$.
\end{lemma}

\begin{proof}
 First, we prove that the set is a sublattice of $L(\kzerox)$. Let
 $U,W\in L(\kzerox)$ with $\pi(U)=\pi(W)=V$. Then $V\subseteq U$ and
 $V\subseteq W$, so $V\subseteq U\cap W$. Thus $V\subseteq \pi(U\cap
 W)\subseteq \pi(U)=V$ and so $\pi(U\cap W)=V$. As well, $\pi(U+W)=
 \pi(U)+\pi(W)=V+V=V$. Thus the set is a sublattice of $L(\kzerox)$.
 Moreover, since $V\subseteq U$, it follows that $\omega(V)=
 V^S\subseteq U$. As $V=\pi(\omega(V))$, we see that $\omega(V)$ is the
 minimum element of the sublattice. Finally, since the sum of all
 $T$-spaces in the set is again a $T$-space in the set, it follows that
 the set has a maximum element, and so is an interval.
\end{proof} 

\begin{lemma}\label{lemma: max goes to max}
 If $U\in M(\kzerox)$, then $\pi(U)\in M(\onekzerox)$.
\end{lemma}

\begin{proof}
 Let $U\in M(\kzerox)$. Since $\pi(U)\subseteq U$, it follows that
 $\pi(U)$ is a proper $T$-space of $\onekzerox$, and thus by
 Lemma \ref{lemma: free have maximal}, there exists
 $W\in M(\onekzerox)$ with $\pi(U)\subseteq W$. Consider
 $\pi(U+\omega(W))=\pi(U)+\pi(\omega(W))=\pi(U)+W \subseteq W$, so
 $U+\omega(W)\ne \kzerox$. Since $U$ was maximal in $\kzerox$, we
 conclude that $U+\omega(W)=U$, so $\omega(W)\subseteq U$. But then
 $W=\pi(\omega(W))
 \subseteq\pi(U)\subseteq W$ and so $W=\pi(U)$, as required.
\end{proof}
 
\begin{proposition}  
 The map $\pi\from L(\kzerox)\to L(\onekzerox)$ induces a bijection from
 $M(\kzerox)$ onto $M(\onekzerox)$, and so every  maximal $T$-space of
 $\kzerox$ is uniquely determined by its one-variable polynomials.
\end{proposition}

\begin{proof}
 Let $U$ be a maximal $T$-space of $\kzerox$, and let
 $V=\omega(\pi(U))$, so $V\subseteq U$. Let $U'$ denote a  maximal
 $T$-space of $\kzerox$ containing $V$, and suppose that $U'\ne U$. Then
 $x\in U+U'$, say $x=f+g$ for  some essential $f\in U$ and essential
 $g\in U'$; that is, $f,g\in \onekzerox$. But then $f\in \pi(U)\subseteq
 V\subseteq U'$ and so $x=f+g\in U'$. However,
 this implies that $U'=\kzerox$, which is not the case.  Thus $U$ is the
 only maximal $T$-space of $\kzerox$ that contains $\pi(U)$. This
 establishes that the restriction of $\pi$ to $M(\kzerox)$ is injective. By
 Lemma \ref{lemma: max goes to max}, $\pi(U)\in M(\onekzerox)$ if $U\in
 M(\kzerox)$, so $\pi$ induces an injective  function from $M(\kzerox)$
 into $M(\onekzerox)$. It remains to prove that $\pi\from M(\kzerox)\to
 M(\onekzerox)$ is surjective. Let $V\in M(\onekzerox)$. By
 Lemma \ref{lemma: one variable guys}, there is a  $T$-space $U$ of $\kzerox$
 that is maximum with respect to the property $\pi(U)=V$. We claim that
 $U\in M(\kzerox)$. For if not, then there exists $W\in M(\kzerox)$ with
 $U\subsetneq W$, and thus $V=\pi(U)\subsetneq \pi(W)$. Since $V$ was
 maximal in $L(\onekzerox)$, it follows that $\pi(W)=\onekzerox$ and so
 $x\in \pi(W)\subseteq W$. But then $W=\kzerox$, which contradicts our
 choice of $W$. Thus $\pi\from M(\kzerox)\to M(\onekzerox)$ is
 surjective.
\end{proof}

 As a result of this observation, we shall focus in the next two
 sections on the study of the maximal $T$-spaces of $\onekzerox$.
 But first, we wish to briefly discuss some questions that remain 
 unanswered at the time of writing. 

 For a given maximal $T$-space $U$ of $\kzerox$, it is not clear how the
 $T$-space generated by $T^{(2)}$ and the  one-variable polynomials in
 $U$ compares to $U$. In general, they will not be equal. For example,
 $\pi(Z)$ is equal to $x^2\,\onekonex$. If $k$ is a finite field of characteristic
 2, we claim that $xy\notin T^{(2)}+\omega(\pi(Z))$. Suppose to the contrary that $xy\in
 T^{(2)}+\omega(\pi(Z))$. Then $xy=\sum_{j} \alpha_ju_j^{i_j}+v$ for some
 $\alpha_j\in k$, $u_j\in \kzerox$, and $v\in T^{(2)}$, where for each $j$, $i_j\ge2$. We
 may assume that $v$ and each $u_j$ depend only on $x$ and $y$. For
 each $j$, if $i_j>2$, then each monomial of $u_j^{i_j}$ has degree at
 least 3. Furthermore, even if $i_j=2$, $xy$ can only appear in $u_j^2$ if $u_j$ has
 linear term $\beta_j x+\gamma_j y$ with $\beta_j,\gamma_j\ne0$. However, for any such $u_j$, 
 $u_j=\beta_j x+\gamma_j y+u_j'$, where each monomial in $u_j'$ has degree at least 2, and in such a case
 (since $k$ has characteristic 2),
 $u_j^2=\beta_j^2x^2+\gamma_j^2y^2+(u_j')^2+\beta_j\gamma_j[x,y]+\beta_j[x,u_j']+\gamma_j[y,u_j']$ with all monomials of
 $(u_j')^2$, $[x,u_j']$, and $[y,u_j']$ having degree at least 3. Let $S$ denote the set of all
 indices $j$ for which $i_j=2$ and $u_j$ has linear term containing both $x$ and $y$. Then

 \begin{align*}
   xy&=\sum_j \alpha_ju_j^{i_j}+v\\
   &=(\sum_{j\in S}\alpha_j\beta_j^2) x^2+(\sum_{j\in S}\alpha_j\gamma_j^2) y^2
   + (\sum_{j\in S}\alpha_j\beta_j\gamma_j) [x,y]\\
   &\hskip40pt +\sum_{j\in S} \alpha_j((u_j')^2+\beta_j[x,u_j']+\gamma_j[y,u_j'])
  +\sum_{j\notin S}\alpha_j u_j^{i_j}+v.\\
 \end{align*}
 As neither $x^2$ nor $y^2$ is a monomial appearing in an element of
 $T^{(2)}$, and each monomial of 
 $\sum_{j\in S} \alpha_j((u_j')^2-\beta_j[x,u_j']-\gamma_j[y,u_j'])
  +\sum_{j\notin S}\alpha_j u_j^{i_j}$ has degree at least 3, it follows that
  $(\sum_{j\in S}\alpha_j\beta_j^2) x^2+(\sum_{j\in S}\alpha_j\gamma_j^2) y^2=0$.
 Thus 
 $$
 xy= (\sum_{j\in S}\alpha_j\beta_j\gamma_j) [x,y] +\sum_{j\in S} \alpha_j((u_j')^2-\beta_j[x,u_j']-\gamma_j[y,u_j'])
  +\sum_{j\notin S}\alpha_j u_j^{i_j}+v.
 $$
 Furthermore, as $xy$ can only appear as a summand in $v$ as a term in $[x,y]$, it follows by the same
 degree considerations that
 $xy=\gamma [x,y]$ for some $\gamma\in k$. As this is not possible, we conclude that $xy\notin
 T^{(2)}+(\pi(Z))^S=\omega(\pi(Z))$, and so
 $T^{(2)}+\omega(\pi(Z))\subsetneq Z$ when $k$ is any finite field of characteristic 2.

 On the other hand, since $2xy=(x+y)^2-x^2-y^2-[y,x]$, 
 $xy\in \omega(\pi(Z))$ when $k$ is a finite field of characteristic $p>2$.

 Furthermore, for any $T$-space $V$ of $\onekzerox$, we might ask how
 the maximum $T$-space $M_V$ in $\kzerox$ that has  image $V$ compares
 to $(\pi^{-1}(V))^S$. In general, we expect $\pi^{-1}(V))^S$ to be larger
 than $M_V$; equivalently, $\pi(\pi^{-1}(V)^S)$ is
 larger than $V$. For example, in $\gftwoz$, consider the $T$-space $V$
 that is generated by $x+x^2$. Then  $x+xy\in \pi^{-1}(V)$, and so $x\in
 \pi^{-1}(V)^S\subseteq\gftwoxz$, which means that 
 $\pi^{-1}(V)^S=\gftwoxz$. However, $V\subseteq \set
 x+x^2\endset^T\subseteq \gftwoz$, and $\set x+x^2\endset^T$ is a
 maximal $T$-ideal of $\gftwoz$.

 \section{A study of maximal $T$-spaces of $\onekzerox$ in the case 
   of a finite field of characteristic $p>2$}

 In this section, $p>2$ is a prime and $k$ is a finite field of
 characteristic $p$ and order $q$.

\begin{definition}\label{definition: def of vn}
 For each $n\ge0$, let $V_n=\set x+x^{q^{2^n}}\endset^S\subseteq \onekzerox$.
\end{definition}

 Since $(\alpha u+\beta v)+(\alpha u+\beta v)^{q^{2^n}}
 =\alpha(u+u^{q^{2^n}})+\beta(v+v^{q^{2^n}})$ for any
 $\alpha,\beta\in k$ and any $u,v\in \kzerox$, it follows that $\set
 x^i+x^{iq^{2^n}}\mid i\ge 1\endset$  is a $k$-linear basis for $V_n$,
 and thus for each  $n\ge0$, $V_n$ is a proper $T$-space of
 $\onekzerox$.

\begin{proposition}\label{proposition: nice for vn}
 Let $n\ge0$. Then $x-x^{q^{2^{n+m}}}\in V_n$ for each $m\ge1$.
\end{proposition}

\begin{proof}
 The proof is by induction on $m$. By definition, $x+x^{q^{2^n}}\in V_n$,
 and so $x^{q^{2^n}}+(x^{q^{2^n}})^{q^{2^n}}=x^{q^{2^n}}+x^{q^{2^{n+1}}}\in V_n$.
 Thus 
 $x-x^{q^{2^{n+1}}}=(x+x^{q^{2^n}})-(x^{q^{2^n}}+x^{q^{2^{n+1}}})\in V_n$, 
 and so the claim holds for $m=1$. Suppose now that $m\ge1$ is such that
 $x-x^{q^{2^{n+m}}}\in V_n$. Then $x^{q^{2^{n+m}}}- (x^{q^{2^{n+m}}})^{q^{2^{n+m}}}
 =x^{q^{2^{n+m}}}- x^{q^{2^{n+m+1}}}\in V_n$, and so $x- x^{q^{2^{n+m+1}}}=
 (x- x^{q^{2^{n+m}}})+(x^{q^{2^{n+m}}}- x^{q^{2^{n+m+1}}})\in V_n$, as required.
\end{proof}
 
\begin{corollary}\label{corollary: sum is all}
 Let $n,m\ge0$ be such that $n\ne m$. Then $V_n+V_m=\onekzerox$.
\end{corollary}

\begin{proof}
 It suffices to prove that for each $n\ge0$ and each $m\ge1$,
 $V_n+V_{n+m}=\onekzerox$.  By Proposition \ref{proposition: nice for vn},
 $x-x^{q^{2^{n+m}}}\in V_n$, and so $2x=(x-x^{q^{2^{n+m}}})+(x+x^{q^{2^{n+m}}})\in V_n+V_{n+m}$.
 Since $2$ is invertible in $k$, it follows that $x\in V_n+V_{n+m}$, and so
 $V_n+V_{n+m}=\onekzerox$.
\end{proof}

\begin{corollary}\label{corollary: p>2 has infinitely many}
 If $k$ is a finite field of characteristic $p>2$, then $\onekzerox$ has
 infinitely many maximal $T$-spaces.
\end{corollary}

\begin{proof}
 For each $n\ge0$, let $Y_n$ denote a maximal $T$-space of $\onekzerox$
 that contains $V_n$. By Corollary \ref{corollary: sum is all}, for 
 $n\ne m$, $Y_n\ne Y_m$.
\end{proof}

\section{A study of maximal $T$-spaces of $\onekzerox$ in the case of a finite field of characteristic 2}

 Let $k$ be a finite field of order $q$ and characteristic 2. Recall
 that $\set x+x^q\endset^T$ is a maximal $T$-ideal and a maximal
 $T$-space of $\onekzerox$.  Our objective is to establish that there
 are infinitely many maximal $T$-spaces of $\onekzerox$, and  we first
 examine the family of $T$-spaces that were used to establish that there
 were infinitely many maximal $T$-spaces of $\onekzerox$ when $k$ was a
 finite field of characteristic $p>2$.

 Recall that for $n\ge0$, $V_n=\set x+x^{q^{2^n}}\endset^S$ in $\onekzerox$.
 In the case $p=2$, we have $q=2^m$ for some positive integer $m$. It is a
 straightforward inductive argument to show that for every integer $i\ge1$,
 $x+x^{2^{im}}\in V_0=W_0$. In particular, $x+x^{2^{m2^n}}\in W_0$ for every
 $n\ge0$, and so $V_n\subseteq W_0$ for every $n\ge0$. 

 Thus we shall need to explore other families of $T$-spaces of
 $\onekzerox$ if we hope to achieve our objective of showing that
 $\onekzerox$ contains infinitely many maximal $T$-spaces.

\begin{definition}\label{definition: gen wn}
 For each positive integer $n$, let $W_n=\set x+x^q,x^{q^n+1}\endset^S$ in $\onekzerox$.
\end{definition}
 
\begin{lemma}\label{lemma: fermat property}
 Let $n\ge1$. Then for any $u,v\in\onekzerox$, 
 $(u+v)^{q^n+1}= u^{q^n+1}+v^{q^n+1}+u^{q^n}v+uv^{q^n}$.
\end{lemma}

\begin{proof}
 We have $(u+v)^{q^n+1}=(u+v)(u+v)^{q^n}=(u+v)(u^{q^n}+v^{q^n})=u^{q^n+1}+v^{q^n+1}
  +u^{q^n}v+uv^{q^n}$.
\end{proof}
  
\begin{definition}\label{definition: def of Lm}
 For each integer $n\ge1$, let $L_n(u,v)=u^{q^n}v+uv^{q^n}$ for each $u,v\in \onekzerox$.
\end{definition}

\begin{proposition}\label{proposition: bilinearity}
 Let $n\ge1$ be an integer. Then $L_n$ is a bilinear function from $\onekzerox\times
 \onekzerox$ to $\onekzerox$.
\end{proposition}

\begin{proof}
 By the symmetry in the definition, it suffices to prove that for every $u_1,u_2,v\in \onekzerox$ and
 $\alpha,\beta\in k$, $L_n(\alpha u_1+\beta u_2,v)=\alpha L_n(u_1,v)+\beta L_n(u_2,v)$. We have
 \begin{align*}
  L_n(\alpha u_1+&\beta u_2,v)=(\alpha u_1+\beta u_2)^{q^n}v+(\alpha u_1+\beta u_2)v^{q^n}\\
           &=(\alpha^{q^n} u_1^{q^n}+\beta^{q^n} u_2^{q^n})v+\alpha u_1v^{q^n}+\beta u_2v^{q^n}\\
           &=\alpha u_1^{q^n}v+\beta u_2^{q^n}v+\alpha u_1v^{q^n}+\beta u_2v^{q^n}\\
           &=\alpha (u_1^{q^n}v+ u_1v^{q^n})+\beta (u_2^{q^n}v+ u_2v^{q^n})\\	   
	   &=\alpha L_n(u_1,v)+\beta L_n(u_2,v).
 \end{align*}
\end{proof}

\begin{proposition}\label{proposition: gen span fermat}
 Let $n\ge1$. Then the set 
 $$
  \set x^i+x^{qi},x^{(q^n+1)i}\mid i\ge1\endset\cup\set x^{q^n i+j} +x^{i+q^n j}\mid i>j\ge1\endset
 $$ 
 is a linear spanning set for $W_n$.
\end{proposition}

\begin{proof}
 Since $W_{n}=\set x+x^q\endset^S+\set x^{q^n+1}\endset^S$ and $\set
 x^i+x^{qi}\mid i\ge1\endset$  is a spanning set for $\set
 x+x^q\endset^S$, it suffices to establish that $\set x^{q^n+1}\endset^S$  
 is spanned by   
 $$   
  S=\set x^{(q^n+1)i}\mid i\ge1\endset \cup \set x^{q^n i+j}+x^{i+q^n j}\mid i>j\ge 1\endset.  
 $$ 
 We first show that $S\subseteq \set x^{q^n+1}\endset^S$. First, we
 observe that for any positive integer $i$, $x^{i(q^n+1)}\in \set
 x^{q^n+1}\endset^S$, and for any $i>j\ge 1$, it follows from Lemma
 \ref{lemma: fermat property} with $u=x^i$ and $v=x^j$ that 
 $x^{q^n i+j} +x^{i+q^n j}\in \set x^{q^n+1}\endset^S$. Thus
 $S\subseteq \set x^{q^n+1}\endset^S$. It remains now to prove that
 $\set x^{q^n+1}\endset^S$ is spanned by $S$. It suffices to prove
 that for every $u\in \onekzerox$, $u^{q^n+1}$ is in the $k$-linear
 span of $S$. We prove this by induction on the number of monomials in $u$.  
 If $u$ is a monomial, the result is immediate.
 Suppose now that $u$ has $t>1$ monomial summands, and the result holds
 for all elements of $\onekzerox$ with fewer than $t$ monomial
 summands. Then $u=v+\alpha x^i$ for some $v\in \onekzerox$ with $t-1$
 monomial summands, and some integer $i\ge1$ and $\alpha\in k^*=k-\set
 0\endset$. By Definition \ref{definition: def of Lm} and Lemma
 \ref{lemma: fermat property}, $u^{q^n+1}=v^{q^n+1}+(\alpha x^i)^{q^n+1}
 +L_n(v,\alpha x^i)=v^{q^n+1}+\alpha^2 x^{(q^n+1)i} +L_n(v,\alpha x^i)$. 
 By the induction hypothesis, $v^{q^n+1}$ is in the linear span of $S$,
 and $x^{(q^n+1)i}\in S$, while by Proposition \ref{proposition:
 bilinearity}, $L_n(v,\alpha x^i)=\alpha L_n(v,x^i)$, so it suffices to
 prove that $L_n(v,x^i)$ is in the linear span of $S$. By Proposition
 \ref{proposition: bilinearity}, it suffices to prove that 
 $L_n(x^j,x^i)$ is in the linear span of $S$ for every $j\ge1$. In fact, 
 $L_n(x^j,x^i)=(x^j)^{q^n}x^i+x^j(x^i)^{q^n}=x^{i+q^nj}+x^{j+q^ni}\in
 S$.
\end{proof}

\begin{corollary}\label{corollary: gen proper}
 For any integer $m\ge1$, $W_{m}$ is a proper $T$-space of $\onekzerox$.
\end{corollary}

\begin{proof}
 Let $n\ge1$, and suppose to the contrary that $W_n=\onekzerox$, so that $x\in W_{n}$. 
 Then by Proposition \ref{proposition: gen span fermat}, 
 $x$ is a linear combination of terms of the form
 $x^i+x^{qi}$, $i\ge1$, $x^{(q^n+1)j}$, $j\ge1$, and  $x^{q^n i+j}
 +x^{i+q^n j}$ where $i>j\ge1$. Suppose that $x=\sum
 \alpha_i(x^i+x^{qi})+ \sum \beta_j x^{(q^n+1)j}
 +\sum \gamma_{i,j}(x^{q^n i+j}+x^{i+q^n j})$, where $\alpha_i,
 \beta_j,\gamma_{i,j}\in k$. Observe that since $(q,q^n+1)=1$, in
 any summand of the form  $x^i+x^{qi}$, $i$ is a multiple of $q^n+1$ if
 and only if $qi$ is a multiple of $q^n+1$. Since we may move any such
 terms to the sum of terms of the form $x^{(q^n+1)j}$, we may assume 
 that in the linear combination  $\sum \alpha_i(x^i+x^{qi})$, no
 monomial of the form $x^{(q^n+1)j}$ appears. Furthermore, $q^n i+j$ is
 a multiple of $q^n+1$ if and only if $i\cong j\mod{q^n+1}$ if and only if
 $i+q^n j$ is a multiple of $q^n+1$, so we may also assume that no summand
 of the form  $x^{q^n i+j}+x^{i+q^n j}$ contains a summand of the form
 $x^{(q^n+1)i}$. Thus $\sum \alpha_i(x^i+x^{qi})+\sum
 \gamma_{i,j}(x^{q^n i+j}+x^{i+q^n j})=x+\sum \beta_j x^{(q^n+1)j}$, 
 where in the sum on the left, there is no monomial of the form
 $x^{(q^n+1)j}$. Thus we must have $\sum \beta_j x^{(q^n+1)j}=0$, and so
 $x=\sum \alpha_i(x^i+x^{qi})+\sum \gamma_{i,j}(x^{q^n i+j}+x^{i+q^n j})$.  
 However, upon evaluation at $x=1$, this yields $1=0$, which is
 not possible. Thus $x\notin W_{n}$.
\end{proof}
 
 In our search for maximal $T$-spaces, we wondered what might be said about $W_n$
 when $n$ is such that $q^n+1$ is prime. This avenue of speculation led us to investigate $W_n$
 for integers $n$ which are the analogue of the Fermat numbers (precisely the
 case when $q=2$). Thus we were led to investigate $W_n$ for positive integers $n$ of the form $q^m$.
 By Corollary \ref{corollary: gen proper}, we know that for any $m\ge0$, $W_{q^m}$ is a proper $T$-space,
 and we consider such to be candidates for maximal $T$-spaces of $\onekzerox$.

\begin{proposition}\label{proposition: gen infinite number}
 Let $n,m$ be nonnegative integers with $n\ne m$. Then $W_{q^n}+W_{q^m}=\onekzerox$.
\end{proposition}

\begin{proof}
 It suffices to consider only $m>n\ge0$, and so we prove that
 for all $n\ge0$ and $t\ge1$, $x^{q^{q^{n+t}}+1}\cong x\mod{W_{q^n}}$.
 Let $n\ge0$, and $t\ge1$. By Proposition \ref{proposition: gen span fermat}, we have 
 $$
  x^{q^{q^n}i+j}\cong x^{i+q^{q^n}j} \mod{W_{q^n}}
 $$
 for every $i,j\ge1$. In particular, when $i=1$ and $j=q^{q^{n+t}-q^n}$, we obtain
 $$
  x^{q^{q^n}+q^{q^{n+t}-q^n}}\cong x^{1+q^{q^n}q^{q^{n+t}-q^n}} \mod{W_{q^n}};
 $$
 that is, $x^{q^{q^n}+q^{q^n(q^t-1)}}\cong x^{1+q^{q^{n+t}}} \mod{W_{q^n}}$. Next, we
 prove that for any integer $a\ge2$, $x^{q^{q^n}+q^{q^na}}\cong
 x^{q^{q^n}+q^{q^n(a-2)}}\mod{W_{q^n}}$.
 We have 
 \begin{align*}\
   x^{q^{q^n}+q^{q^na}} &= x^{q^{q^n}+q^{q^n}q^{q^n(a-1)}}=x^{q^{q^n}(1+q^{q^n(a-1)})}\\
                        &\cong x^{1+q^{q^n(a-1)}}= x^{1+q^{q^n}q^{q^n(a-2)}}\quad\text{since
                        $x\cong x^q \mod{W_{q^n}}$}\\ 
			&\cong x^{q^{q^n}+q^{q^n(a-2)}}\mod{W_{q^n}}.
 \end{align*}
 We now apply this result iteratively, starting with $a=q^t-1$, an odd integer, drawing the conclusion that
 $$
   x^{q^{q^n}+q^{q^n(q^t-1)}}\cong x^{q^{q^n}+q^{q^n}}=(x^{q^{q^{n}}})^2\cong x^2\mod{W_{q^n}}.
 $$  
 Thus we have established that $x^{1+q^{q^{n+t}}}\cong x^{q^{q^n}+q^{q^n(q^t-1)}}
 \cong x^2 \mod{W_{q^n}}$. Since $x^{1+q^{q^{n+t}}}\in W_{q^{n+t}}$, we obtain 
 that $x^2\in W_{q^n}+W_{q^{n+t}}$.  Now, $q=2^s$ for some $s\ge1$, and thus we 
 have $x^q=(x^{2^{s-1}})^2\in W_{q^n}+W_{q^{n+t}}$. Finally, as $x+x^q\in 
 W_{q^n}+W_{q^{n+t}}$, we have $x\in W_{q^n}+W_{q^{n+t}}$, as required.
\end{proof}


\begin{corollary}\label{corollary: gen inf many max t-spaces}
 There are infinitely many maximal $T$-spaces of $\onekzerox$ that contain $W_0$.
\end{corollary}

\begin{proof}
 By Corollary \ref{corollary: gen proper}, $W_{q^n}$ is a proper $T$-space for 
 every $n\ge0$. For each $n\ge0$, let $M_n$ denote some maximal $T$-space containing
 $W_{q^n}$. Now, let $m,n\ge0$ with $m\ne n$, and suppose that
 $M_m=N_n$. Then by Proposition \ref{proposition: gen infinite number}, 
 we would have $M_n=M_n+M_m=\onekzerox$,
 contradicting the fact that $M_n$ is a maximal $T$-space of $\onekzerox$.
\end{proof}

 We have not yet addressed the question as to whether or not
 $W_{q^n}$ is itself maximal. We shall investigate this issue now,
 but only in the case where $q=2$. To begin with, we shall study $W_{2^0}=W_1$.
 As a consequence of Proposition \ref{proposition: gen span fermat}, we 
 know that $W_1$ is a proper $T$-space of $\gftwoxz$.

\begin{proposition}\label{proposition: w1}
 $W_1$ is a maximal $T$-space of $\gftwoz$.
\end{proposition}

\begin{proof}  
 Let $f\in \gftwoz-W_1$. Since $x^i\cong x^{2i}\mod{W_1}$ for every
 positive integer $i$, we may assume that $f$ has no monomial summands
 of even degree. Furthermore, observe that $(x+x^2)^3\in W_1$, and since
 $(x+x^2)^3=x^3+x^4+x^5+x^6$ and $x^3, x^6\in W_1$, it follows that
 $x^4+x^5\in W_1$. Thus $x^5\cong x^4\cong x^2\cong x\mod{W_1}$. As
 well, for every integer $n\ge2$, we have $(x(x+x^n))^3\in W_1$, so 
 $x^3(x^3+x^{n+2}+x^{2n+1}+x^{3n})=x^6+x^{n+5}+x^{2n+4}+x^{3n+3}\in
 W_1$, and thus for every integer $n\ge2$, $x^{n+5}+x^{2(n+2)}\in W_1$.
 But then $x^{n+5}\cong x^{2(n+2)}\cong x^{n+2}\mod{W_1}$ for every
 integer $n\ge2$. That is; for every integer $n\ge7$, $x^n\cong
 x^{n-3}\mod{W_1}$. It follows now that in $f$, every monomial of odd
 degree greater than or equal to 7 can be replaced by one of odd degree
 at most 5. Finally, since $x^3\in W_1$, we may assume that $f$ does not
 have $x^3$ as a summand, and since $x^5\cong x\mod{W_1}$, we may assume
 that $f$ does not have $x^5$ as a summand. Thus $f=x$, and so $W_1+\set
 f\endset^S=\gftwoz$. 
\end{proof}

 Next, we study $W_2$. Again, as a result of Proposition \ref{proposition: gen span fermat}, we know that $W_2$
 is a proper $T$-space of $\gftwoz$.

\begin{proposition}
 $W_2$ is a maximal $T$-space of $\gftwoz$, and moreover, $x^7\notin W_2$.
\end{proposition}

\begin{proof}  
 Let $f\in \gftwoz-W_2$. Since $x^i\cong x^{2i}\mod{W_2}$ for every
 positive integer $i$, we may assume that $f$ has no monomial summands
 of even degree. Furthermore, since for every $j>i\ge1$, $x^{i+4j}\cong
 x^{4i+j}$ and $i+4j>4i+j$, and every odd integer greater than 16 can be
 written in the form $i+4j$ for some $0<i<4\le j$, it follows that every
 monomial in $f$ of (odd) degree greater than $16$ can be reduced to an
 odd degree less than $16$. As well, $13=4(3)+1$ and $9=4(2)+1$, so
 $x^{13}\cong x^{7}\mod{W_2}$ and $x^{9}\cong x^{6}\cong x^3\mod{W_2}$. 
 Moreover, $11=4(2)+3$, so $x^{11}\cong x^{14}\cong x^7\mod{W_2}$. Thus
 (since $x^5\cong x^{15}\cong 0\mod{W_2}$) we may assume  that $f$ is a
 sum of monomials in  $\set x,x^3,x^7\endset$. Furthermore, we have
 $19=4(4)+3$, $23=4(5)+3$, $4(3)+5=17=4(4)+1$, $27=4(6)+3$, and $31=4(7)+3$, 
 so $x^{19}\cong x^{16}\cong x\mod{W_2}$, $x^{23}\cong x^{17}\cong x^8\cong
 x\mod{W_2}$, $x^{27}\cong x^{18}\cong x^9\cong x^3 \mod{W_2}$,
 and $x^{31}\cong x^{19}\cong x\mod{W_2}$. Finally, $21=4(5)+1$ and so
 $x^{21}=x^{4(5)+1}\cong x^9\cong x^3\mod{W_2}$. We shall apply these
 observations as needed below.
 
 \noindent Case 1: $f=x^3$. Observe that $W_1+\set x^5\endset^S=
 W_2+\set x^3\endset^S$. It was observed in the proof of
 Proposition \ref{proposition: w1} that $x^5\cong x\mod{W_1}$, so $x^5\notin W_1$. 
 By Proposition \ref{proposition: w1}, $W_1+\set x^5\endset^S=\gftwoz$, so
 $W_2+\set x^3\endset^S=\gftwoz$.
 
 \noindent Case 2: $f=x^7$. We have $0\cong
 (x+x^2)^7=x^7+x^8+x^9+x^{10}+x^{11}+x^{12}+x^{13}+x^{14}
 \cong x\mod{W_2+\set x^7\endset^S}$, so $x\in
 W_2+\set x^7\endset$. In particular, $x^7\notin W_2$.
 
 \noindent Case 3: $f=x+x^3$. Then $0\cong (x+x^2)+(x+x^2)^3 =
 (x+x^2)+x^3+x^4+x^5+x^6\cong x^4\cong x\mod{W_2+\set
 x+x^3\endset^S}$, so $x\in W_2+\set x+x^3\endset^S$.
 
 \noindent Case 4: $f=x+x^7$. We have $(x+x^2)+(x+x^2)^7\in W_2+\set
 x+x^7\endset^S$, and since $x+x^2\in W_2$, it follows that
 $(x+x^2)^7\in W_2+\set x+x^7\endset^S$. As $(x+x^2)^7=x^7+x^8+x^9+x^{10}+x^{11}+x^{12}+x^{13}+x^{14}
 \cong x\mod{W_2}$, it follows that $x\in W_2+\set x+x^7\endset^S$.
 
 \noindent Case 5: $f=x^3+x^7$. Modulo $W_2+\set
 x^3+x^7\endset^S$, we have
 \begin{align*}
   0&\cong (x+x^5)^3+(x+x^5)^7\\
   &=x^3+x^7+x^{11}+x^{15}+x^7+x^{11}
                       +x^{15}+x^{19}+x^{23}+x^{27}+x^{31}+x^{35}\\
   &\cong x^7+x+x+x^{3}+x\cong  x.
 \end{align*}
 Thus $x\in W_2+\set x^3+x^7\endset^S$.
 
 \noindent Case 6: $f=x+x^3+x^7$. Then $(x+x^3)+(x+x^3)^3+(x+x^3)^7\in
 W_2+\set f\endset^S$. Modulo $W_2$, we have 
 $(x+x^3)+(x+x^3)^3+(x+x^3)^7=
 x+x^3+x^3+x^5+x^7+x^9+x^7+x^9+x^{11}+x^{13}+x^{15}+x^{17}+x^{19}+x^{21}  
 =x+x^5+x^{11}+x^{13}+x^{15}+x^{17}+x^{19}+x^{21}\cong
 x+x^{11}+x^{13}+x^{17}+x^{19}+x^{21}\cong x+x^3\mod{W_2}$, and so $x+x^3\in W_2+\set
 f\endset^S$. But then $W_2+\set x+x^3\endset^S\subseteq W_2+\set
 f\endset^S$, and so by Case 3, $W_2+\set f\endset^S=\kzerox$.

 This completes the case-by-case analysis, and thus $W_2$ is a maximal $T$-space.  
\end{proof}

 While we have not yet determined the status of $W_{2^n}$ for
 $n>1$, we do know that it is not necessarily the case that the
 $T$-space $\set x+x^2,x^p\endset^S$ is
 proper for every prime $p$. In fact, as we now show, $\set x+x^2,x^7\endset^S=\gftwoxz$. We
 remark that since $q=2$ in this discussion, $q$-homogeneity is a
 non-condition since $q-1=1$.
 
 For convenience, we shall let $P=\set x+x^2,x^7\endset^S$.

 For any $i,j\ge1$, $(x^i+x^j)^7-x^{7i}-x^{7j}\in P$. Since $\choice 7,t
 \cong 1\mod{2}$ for every $t$ with $0\le t\le 7$, we have 
 \[
   x^{6i+j}+x^{5i+2j}+x^{4i+3j}+x^{3i+4j}+x^{2i+5j}+x^{i+6j}\in P\tag{1}
 \]
 for all $i,j\ge1$. In (1), set $i=j+1$ to obtain
 \[
  x^{7j+6}+x^{7j+5}+x^{7j+4}+x^{7j+3}+x^{7j+2}+x^{7j+1}\in P\tag{2}
 \]
 for all $j\ge1$. Then in (2), set $j=1$, $j=2$, $j=3$, and $j=4$, 
 respectively, and use everywhere possible the fact that for every $t\ge1$,
 $x^t\cong x^{2t}\mod{P}$ to obtain
 \begin{gather}
  x+x^3+x^5+x^9+x^{11}+x^{13}\in P\tag{3}\\
  x+x^5+x^9+x^{15}+x^{17}+x^{19}\in P\tag{4}\\
  x^3+x^{11}+x^{13}+x^{23}+x^{25}+x^{27}\in P\tag{5}\\
  x+x^{15}+x^{17}+x^{29}+x^{31}+x^{33}\in P\tag{6}  
 \end{gather}
 Next, set $i=j+2$ in (1) to obtain
 \[
  x^{7j+12}+x^{7j+10}+x^{7j+8}+x^{7j+6}+x^{7j+4}+x^{7j+2}\in P\tag{7}
 \]
 for all $j\ge1$, then in (7), set $j=1$ and $j=3$, respectively, to obtain
 \begin{gather}
  x^9+x^{11}+x^{13}+x^{15}+x^{17}+x^{19}\in P\tag{8}\\
  x^{23}+x^{25}+x^{27}+x^{29}+x^{31}+x^{33}\in P\tag{9}
 \end{gather}
 From (4) and (8), we obtain that
 \[
  x+x^5+ x^{11}+x^{13}\in P\tag{10}
 \]
 and then from (3) and (10) we get
 \[ 
  x^3+x^9\in P.\tag{11}
 \]
 As well, from (5) and (9) we obtain
 \[
  x^3+x^{11}+x^{13}+x^{29}+x^{31}+x^{33}\in P\tag{12}
 \]
 and then from (6) and (12) we get
 \[
  x+x^3+x^{11}+x^{13}+x^{15}+x^{17}\in P.\tag{13}
 \]
 Then (13) and (4) gives
 \[
  x^3+x^5+x^9+x^{11}+x^{13}+x^{19}\in P.\tag{14}
 \]
 By (11), we have $x^3+x^9\in P$, so from (14) we now obtain
 \[
  x^5+x^{11}+x^{13}+x^{19}\in P.\tag{15}
 \]
 Now (15) and (10) yield
 \[
  x+x^{19}\in P,\tag{16}
 \]
 so $x\cong x^{19}\mod{P}$.
 Now from (11), we have 
 $$
 (x^i+x^j)^3+(x^i+x^j)^9-x^{3i}-x^{3j}-x^{9i}-x^{9j}\in P,
 $$
 and so
 \[
  x^{2i+j}+x^{i+2j}+x^{8i+j}+x^{i+8j}\in P\tag{17}
 \]
 for all $i,j\ge1$. Set $i=j+1$ in (17) to get
 \[
  x^{3j+2}+x^{3j+1}+x^{9j+8}+x^{9j+1}\in P\tag{18}
 \]
 for all $j\ge1$. In (18), set $j=2$ and $j=3$, respectively, to get (note that $x^7,x^{35}\in P$)
 \begin{gather}
  x+x^{13}+x^{19}\in P\tag{19}\\
  x^{5}+x^{11}\in P.\tag{20}
 \end{gather}
 From (16) and (19), we obtain $x^{13}\in P$, and this, together with (20) and (10), gives $x\in P$.
 
 Thus for $p=3,5$ (the first two Fermat primes), $W_1=\set x+x^2,x^{3}\endset^S$
 and $W_2=\set x+x^2,x^{5}\endset^S$ are maximal
 $T$-spaces, while for 7, the first odd non-Fermat prime, we have $\set x+x^2,x^{7}\endset^S=\gftwoxz$.
 There are many interesting questions that arise from this exploration.
 For example, is it true that $\set x+x^2,x^p\endset^S$ is a maximal $T$-space for every
 Fermat prime $p$? If so, are the Fermat primes the only primes for which
 $\set x+x^2,x^p\endset^S$ is maximal? For $n>1$, if $W_{2^n}$ is not maximal, can one
 describe the maximal $T$-spaces that contain it?

\section{Summary of the nonunitary case}

\begin{theorem}
 For any field $k$, and any nonempty set $X$, the following hold.
  \begin{list}{(\roman{parts})}{\usecounter{parts}} 
  \item $Z$ is a maximal $T$-ideal of $\kzerox$, and if $k$ is infinite, $Z$ is the 
	maximum $T$-ideal of $\kzerox$. If $k$ is finite of order $q$, 
	then $\kzerox$ has exactly one other maximal $T$-ideal; namely $W_0=T^{(2)}+\set x-x^q\endset^T$.
  \item Every maximal $T$-ideal of $\kzerox$ is a maximal $T$-space of $\kzerox$.
  \item If $k$ is infinite, then $Z$ is the only maximal $T$-space of $\kzerox$.
  \item If $k$ is finite, then $\kzerox$ has infinitely many maximal $T$-spaces.
  \end{list}
\end{theorem}

\begin{proof}

 (i) was proven in Theorem 3 of \cite{S} for the case when $X$ is infinite, and in 
 Corollary \ref{corollary: gen of [S] z result} when $X$ is finite and $k$ is infinite,
 and in Corollary \ref{corollary: finite field version} when both $X$ and $k$ are finite.
 (ii) follows from Proposition \ref{proposition: max t-ideals are max t-spaces}, and (iii) follows 
 from Proposition \ref{proposition: infinite field t-space}. Finally, 
 (iv) follows from Corollary \ref{corollary: bijective mapping on maximal t-spaces} together 
 with Corollary \ref{corollary: p>2 has infinitely many} for
 the case of characteristic $p>2$, and by Corollary
 \ref{corollary: gen inf many max t-spaces} for the case of characteristic 2.
\end{proof}  
 
\section{The unitary case}

 Let $k$ be an infinite field, and let $X$ be a nonempty set. 
 Then $\konex$ has a maximum $T$-ideal;
 namely $T^{(2)}$. Set  $Y=T^{(2)}+\set x^{char(k)}\endset^S$, where
 in the characteristic zero case, we interpret $x^0$ as $1$. Note that
 in every case we have $k\subseteq Y$.

\begin{proposition}
 If $k$ is an infinite field, then $Y$ is a maximum proper $T$-space of $\konex$.
\end{proposition}

\begin{proof}
 We are to prove that every proper $T$-space of $\konex$ is 
 contained in $Y$. Suppose that $V$ is a $T$-space  of 
 $\konex$ that is not contained in $Y$. Let $f\in V-Y$ be 
 essential (that is, every variable that appears in any 
 monomial of $f$ appears in every monomial of $f$), say 
 on variables $x_1,x_2,\ldots,x_t$. Since $k$ is  infinite,
 each homogeneous component of $f$ belongs to $V$, and so we 
 may assume that  $f=\alpha x_1^{i_1}\cdots x_t^{i_t}
 +u(x_1,x_2,\ldots,x_t)$ for some positive integers $i_1,i_2,
 \ldots,i_t$ and some $u(x_1,\ldots,x_t)\in T^{(2)}$. If 
 $char(k)=p>0$, and every $i_j$ divisible by $p$, then $f\in Y$, 
 which is not the case. Thus if $char(k)=p>0$, then there exists
 $j$ such that $i_j$ is not divisible by $p$. On the other hand, if
 $char(k)=0$, then we shall choose $j=1$. Set $x_{i_r}=1$ for each
 $r\ne j$. Since $u(1,1,\ldots,1,x_j,1,\ldots,1)=0$, it follows that
 $x_j^{i_j}\in V$. Let $n=i_j$. As $V$ is a $T$-space,
 it follows that $(x_j+1)^n\in V$. Since $k$ is infinite, every homogeneous
 component of $(x_j+1)^n=\sum_{i=1}^n \choice n,i x_j^i$ also belongs to $V$.
 In particular, $\choice n,1 x_j\in V$, and since $\choice n,1=n\ne0$ in $k$, we 
 conclude that $x_j\in V$. Thus $V=\konex$, which proves that every proper 
 $T$-space of $\konex$ is contained in $Y$, as required.
\end{proof}

 Thus every $T$-ideal of $\konex$ is contained in $Y$, and $Y$ is not
 a $T$-ideal of $\konex$ (since $k\subseteq Y$).

 Now suppose that $k$ is a finite field of order $q$ and characteristic
 $p$. In this case, $\konex$ has maximum $T$-ideal equal to
 $T^{(2)}+\set x-x^q\endset^T$. As in the preceding infinite field case,
 the maximum $T$-ideal is not a maximal $T$-space, as $k+T^{(2)}+\set
 x-x^q\endset^T$ is a proper $T$-space containing it.

\begin{proposition} 
 If $k$ is a finite field of order $q$, then $W=k+T^{(2)}+\set x-x^q\endset^T$ 
 is a maximal $T$-space of $\konex$.
\end{proposition}

\begin{proof}
 Let $f\notin W$, and let $U=W+\set f\endset^S$. We prove that
 $U=\konex$. Note that 
 $x^i\cong x^{q+i-1}\mod{W}$ for every positive integer
 $i$. Thus we may assume that in every monomial $u$ of $f$, each
 variable has degree at most $q-1$. We may also assume that $f$ is
 $q$-homogeneous, and thus $f$ is a monomial (since $T^{(2)}\subseteq W$). 
 Choose one variable that appears in $u$ and set all other variables equal 
 to 1 to obtain that for some $x\in X$ and some positive
 integer $i$, $x^i\in U$. Let $i=p^tm$ where 
 $(p,m)=1$. Then
 $(x+1)^i=(x^{p^t}+1)^m\in U$, and so each homogeneous component of  
 $(x^{p^t}+1)^m$ belongs to $U$ as well. In particular, $mx^{p^t}\in U$, and
 since $m\ne 0$ in $k$, we have $x^{p^t}\in U$. But then for every $j$,
 $x^{jp^t}\in U$. Choose $j$ such that $jp^t=q^r$ for some positive
 integer $r$. Then $x^{q^r}=x^{jp^t}\in U$. But $x^{q^r}\cong
 x\mod{W}$ and thus $x^{q^r}\cong x\mod{U}$, which means that $x\in U$.
\end{proof} 

 In the discussion of the unitary case $\konex$, we shall frequently consider
 $U\subseteq \kzerox$ and wish to compare the $T$-space generated by $U$ in
 $\kzerox$, which we shall now denote by $U^{S_0}$, to the $T$-space
 generated by $U$ in $\konex$, which we shall denote by $U^S$.

 In Definition \ref{definition: def of vn}, for each $n\ge0$, we defined $V_n=\set
 x+x^{q^{2^n}}\endset^{S_0}\subseteq \onekzerox$. 
 
\begin{definition}  
 In $\konex$, for each $n\ge1$, let $\unitaryv{n}=T^{(2)}+\set x+x^{q^{2^n}}\endset^{S}$. 
\end{definition}

 Note that for each $n\ge0$, $\unitaryv{n}=k+T^{(2)}+\set x+x^{q^{2^n}}\endset^{S_0}$, and so
 $\unitaryv{n}$ is a proper $T$-space of $\konex$.
 
\begin{proposition}\label{proposition: p>2 k<X> has infinitely many maxl tspaces}
 Let $k$ be a finite field of order $q$ and characteristic $p>2$. Then
 for any $m,n\ge0$ with $m\ne n$, $\unitaryv{m}+\unitaryv{n}=\konex$.
\end{proposition}

\begin{proof}  
 By Corollary \ref{corollary: sum is all}, for any positive
 integers $m,n$ with $m\ne n$, in $\onekzerox$ we have $V_n+V_m=\onekzerox$
 and so $\unitaryv{m}+\unitaryv{n}=\konex$. 
\end{proof} 

The proof of the following corollary is similar to the corresponding result in the
nonunitary case and is therefore omitted.

\begin{corollary}\label{corollary: unitary case, p>2}
 If $k$ is a finite field of characteristic $p>2$, then $\konex$ has infinitely 
 many maximal $T$-spaces.
\end{corollary}

 It remains to examine the situation when $p=2$. Assume now that $k$ is
 a field of order $q$ and characteristic $2$. Recall that in
 Definition \ref{definition: gen wn} for each positive integer $n$, we have defined
 $W_n=\set x+x^q,x^{q^n+1}\endset^{S_0}$ in $\onekzerox$.

\begin{proposition}\label{proposition: wn is wns}
 For each $n\ge1$, $W_{n}^{S}=k+W_{n}^{S_0}$.
\end{proposition} 

\begin{proof}
 Let $n\ge1$. Evidently, we have $k+W_{n}^{S_0}\subseteq W_{n}^{S}$.
 For any $\alpha\in k$ and any $u\in \kzerox$, we have $(\alpha+u)
 +(\alpha+u)^q=\alpha+u+\alpha^q+u^q=2\alpha+(u+u^q)\in k+ W_n^{S_0}$, and
 $(\alpha+u)^{q^n+1}=(\alpha+u)^{q^n}(\alpha+u)=(\alpha+u^{q^n})(\alpha+u)
 =\alpha^2+\alpha u+\alpha u^{q^n}+u^{q^n+1}$. 
 Now $\alpha u+\alpha u^{q^n}\in \set x+x^q\endset^{S_0}$,
 and $u^{q^n+1}\in \set x^{q^n+1}\endset^{S_0}$, 
 so $(\alpha+u)^{q^n+1}\in k+W_{n}^{S_0}$. Thus  
 $W_{n}^{S}\subseteq k+W_{n}^{S_0}$, and so equality prevails. 
\end{proof} 

\begin{definition} 
 For each positive integer $n$, let $\unitaryw{n}=W_{n}^S$ in $\konex$.
\end{definition}

 By Corollary \ref{corollary: gen proper}, for any integer $n\ge1$, $W_{n}$ is a
 proper $T$-space of $\onekzerox$, and thus $\unitaryw{n}$ is a proper
 $T$-space of $\konex$. In particular, for each $n\ge0$,
 $\unitaryw{q^n}$ is a proper $T$-space of $\konex$.

\begin{proposition}\label{proposition: unitary gen infinite number}
 Let $n,m$ be nonnegative integers with $n\ne m$. Then $\unitaryw{q^n}+\unitaryw{q^m}=\konex$.
\end{proposition}

\begin{proof}
 By Proposition \ref{proposition: wn is wns}, $\unitaryw{q^n}+\unitaryw{q^m}=
 k+{W}_{q^n}^{S_0}+k+{W}_{q^m}^{S_0}= k+{W}_{q^n}^{S_0}+{W}_{q^m}^{S_0}$, and 
 by Proposition \ref{proposition: gen infinite number}, ${W}_{q^n}^{S_0}+{W}_{q^m}^{S_0}=\kzerox$,
 so $\unitaryw{q^n}+\unitaryw{q^m}=k+\kzerox=\konex$.
\end{proof}
  
\begin{corollary}  
 Let $k$ be a finite field of characteristic 2. Then $\konex$ 
 has infinitely many maximal $T$-spaces.
\end{corollary}

\begin{proof} 
 Let $k$ have order $q$. We have observed above that for each $n\ge0$,
 $\unitaryw{q^n}$ is a proper $T$-space of $\konex$, and by
 Proposition \ref{proposition: unitary gen infinite number}, for $m\ne n$, no maximal
 $T$-space of $\konex$ contains both $\unitaryw{q^m}$ and
 $\unitaryw{q^n}$. Thus $\konex$ has infinitely many maximal $T$-spaces.
\end{proof}

\end{document}